\documentclass[11pt,  reqno]{amsart}
\usepackage{srcltx}
\usepackage{eurosym}
\usepackage{mathtools}
\usepackage{amsmath}
\usepackage{amsfonts}
\usepackage{amssymb}
\usepackage{amsthm}
\usepackage{graphicx}
\usepackage{mathrsfs}
\usepackage{xcolor}
\usepackage{exscale}
\usepackage{latexsym}
\usepackage{esvect}
\usepackage[T1]{fontenc}
\usepackage{enumitem}
\setlist[description]{noitemsep, topsep=2pt}
\setlist[enumerate]{noitemsep, topsep=0pt}

 \usepackage[margin=3cm]{geometry}
\usepackage[colorlinks,plainpages=true,pdfpagelabels,hypertexnames=true,colorlinks=true,pdfstartview=FitV,linkcolor=blue,citecolor=red,urlcolor=black]{hyperref}
\PassOptionsToPackage{unicode}{hyperref}
\PassOptionsToPackage{naturalnames}{hyperref}
\usepackage[ddmmyyyy]{datetime}
\numberwithin{equation}{section}
\everymath{\displaystyle}
\usepackage[capitalize,nameinlink]{cleveref}
\crefname{section}{Section}{Sections}
\crefname{subsection}{Subsection}{Subsections}
\crefname{condition}{Condition}{Conditions}
\crefname{hypothesis}{Hypothesis}{Conditions}
\crefname{assumption}{Assumption}{Assumptions}
\crefname{lemma}{Lemma}{Lemmas}
\crefname{claim}{Claim}{Claims}
\crefname{conclusion}{Conclusion}{Conclusions}

\crefformat{equation}{\textup{#2(#1)#3}}
\crefrangeformat{equation}{\textup{#3(#1)#4--#5(#2)#6}}
\crefmultiformat{equation}{\textup{#2(#1)#3}}{ and \textup{#2(#1)#3}}
{, \textup{#2(#1)#3}}{, and \textup{#2(#1)#3}}
\crefrangemultiformat{equation}{\textup{#3(#1)#4--#5(#2)#6}}%
{ and \textup{#3(#1)#4--#5(#2)#6}}{, \textup{#3(#1)#4--#5(#2)#6}}%
{, and \textup{#3(#1)#4--#5(#2)#6}}

\Crefformat{equation}{#2Equation~\textup{(#1)}#3}
\Crefrangeformat{equation}{Equations~\textup{#3(#1)#4--#5(#2)#6}}
\Crefmultiformat{equation}{Equations~\textup{#2(#1)#3}}{ and \textup{#2(#1)#3}}
{, \textup{#2(#1)#3}}{, and \textup{#2(#1)#3}}
\Crefrangemultiformat{equation}{Equations~\textup{#3(#1)#4--#5(#2)#6}}%
{ and \textup{#3(#1)#4--#5(#2)#6}}{, \textup{#3(#1)#4--#5(#2)#6}}%
{, and \textup{#3(#1)#4--#5(#2)#6}}

\crefdefaultlabelformat{#2\textup{#1}#3}

\newtheorem{thm}{Theorem}[section]
\newtheorem{lem}[thm]{Lemma}
\newtheorem{rem}[thm]{Remark}

\numberwithin{equation}{section}

\newcommand{\R}{\mathbb{R}}
\newcommand{\N}{\mathbb{N}}

\newcommand{\Om}{\Omega}
\newcommand{\hh}{\mathcal{H}}
\let\oldtocsection=\tocsection
\let\oldtocsubsection=\tocsubsection
\let\oldtocsubsubsection=\tocsubsubsection
\renewcommand{\tocsection}[2]{\hspace{0em}\oldtocsection{#1}{#2}}
\renewcommand{\tocsubsection}[2]{\hspace{1em}\oldtocsubsection{#1}{#2}}
\renewcommand{\tocsubsubsection}[2]{\hspace{2em}\oldtocsubsubsection{#1}{#2}}
\newcommand{\redref}[2]{\texorpdfstring{\protect\hyperlink{#1}{\textcolor{black}{(}\textcolor{red}{#2}\textcolor{black}{)}}}{}}
\newcommand{\redlabel}[2]{\hypertarget{#1}{\textcolor{black}{(}\textcolor{red}{#2}\textcolor{black}{)}}}
\newcommand{\descitem}[2]{\item[\textcolor{black}{(#1)}] \label{#2}}
\newcommand{\descref}[2]{\hyperref[#1]{\textcolor{black}{(}\textcolor{blue}{\bf #2}\textcolor{black}{)}}}

\newcommand{\descitemnb}[2]{\item[\textcolor{black}{#1}] \label{#2}}

\makeatletter
\g@addto@macro\normalsize{%
  \setlength\abovedisplayskip{3pt}
  \setlength\belowdisplayskip{3pt}
  \setlength\abovedisplayshortskip{1pt}
  \setlength\belowdisplayshortskip{3pt}
}
\begin{document}
\title[]
 {Twice differentiability of solutions to fully nonlinear parabolic equations near the boundary}

 \author{Karthik Adimurthi}
 \address{Tata Institute of Fundamental Research\\
Centre For Applicable Mathematics \\ Bangalore-560065, India}\email[Karthik Adimurthi]{karthikaditi@gmail.com \and kadimurthi@tifrbng.res.in}

 \author{Agnid Banerjee}
\address{Tata Institute of Fundamental Research\\
Centre For Applicable Mathematics \\ Bangalore-560065, India}\email[Agnid Banerjee]{agnidban@gmail.com}

\author{Ram Baran Verma}
\address{Tata Institute of Fundamental Research\\
Centre For Applicable Mathematics \\ Bangalore-560065, India}\email[Ram Baran Verma]{rambv88@gmail.com}

 \thanks{Second author is supported in part by SERB Matrix grant MTR/2018/000267}


\begin{abstract}
In this paper, we prove $\hh^{2+\alpha}$ regularity for viscosity solutions to non-convex fully nonlinear parabolic equations  near the boundary.  This constitutes the parabolic  counterpart   of  a similar $C^{2, \alpha}$ regularity result  due to Silvestre and Sirakov  proved in \cite{SS}  for solutions  to  non-convex fully nonlinear elliptic equations. 

\end{abstract}

\maketitle

\tableofcontents

\section{Introduction}
In this paper, we study the boundary regularity of viscosity solutions to fully nonlinear parabolic equations of the type
\[
F(D^2u, Du, x, t)= f(x,t),
\]
with appropriate  structural assumptions  on $F$ as defined in \cref{section2}. In order to put things in the right historical perspective, we note that recently in \cite[Theorem 1.3]{SS}, Silvestre and Sirakov  proved that  viscosity solutions to fully  nonlinear elliptic equations of the type 
\begin{equation*}
\begin{cases}
F(D^2 u,Du, x) = f(x)\quad \text{in $\Omega$},
\\
u=g\quad \text{on $\partial \Omega$},
\end{cases}
\end{equation*}
are $C^{2, \alpha}$ in a small neighborhood of  the boundary provided $F,f$ and the boundary conditions are sufficiently regular.  We note that solutions to the above equations do not have the same regularity in the interior of the domain because a counterexample due to Nadirashvili and Vladut \cite{NV}, which shows that the best apriori regularity available for solutions to such equations  is only $C^{1,\alpha}$ even in the case when $F$ is smooth. 

In the case when $F$ is additionally concave (or convex),  we  note that the $C^{2, \alpha}$  interior regularity result is the well-known Evans-Krylov theorem proved in \cite{Ev,Kr} (see also \cite{CC}). Therefore the  regularity result in \cite{SS}  demonstrates that the singularities   can only occur \emph{"far"} from the boundary.   
The proof  of the boundary regularity  in \cite{SS} relies on first showing a $C^{2, \alpha}$ type decay at boundary points. The authors subsequently combine  such a decay estimate with a "regularity under smallness" result of Savin in \cite{sa} (more precisely,  using a  certain generalization of Savin's result  proved in \cite{dT}), to obtain $C^{2}$ regularity in a  neighborhood of the boundary.   

We  note that such  a  $C^{2,\alpha}$ type  boundary decay estimate is first established for homogeneous equations. This  is done using the  Krylov type  boundary $C^{1, \alpha}$  estimates as in \cite{Kr1} (which in \cite{SS} is established for viscosity solutions as in Theorem 1.1 of their paper) applied to the  directional derivatives which belong to the Pucci class which crucially uses the translation invariance of the PDE,   followed by  a clever  argument based on barriers. The passage from the homogeneous to the more general equations is then based on a perturbation argument which  uses a "quantitative" approximation lemma (more specifically, see \cite[Lemma 4.2]{SS}).     We note that such a boundary regularity result,  besides being of independent interest has also found  interesting applications to  unique continuation of  fully nonlinear elliptic equations as in \cite{AS}. See also \cite{Ba} for further generalizations in this direction.  Moreover the boundary regularity result also plays an important role  in the context  of overdetermined boundary value problems for fully nonlinear equations  as in \cite{SS1}. 

\medskip

In this paper, we establish the  parabolic analogue of the  such a  regularity result  of \cite{SS} as proved in   \cref{mn1}, \cref{mn2} and \cref{mn4}. Similar to the elliptic case, the proof  of our main result is based on first establishing an appropriate $\hh^{2+\alpha}$ decay at the boundary for the homogeneous equations by a suitable  adaptation of the ideas from \cite{SS} to the parabolic case. The passage to the more general equations is however different from that in \cite{SS} and is instead  based on a compactness argument inspired by the ideas from the fundamental paper of Caffarelli, see \cite{Ca}.  Such a $\hh^{2+\alpha}$ type decay is then combined with the "flatness implies smoothness" result of Savin, proved for the parabolic analogue by  Yu Wang in \cite{Wa}.  Finally under suitable compatibility conditions at the "corner" points of the parabolic boundary (a result due to Lihe Wang in \cite{W1}), which guarantees  a similar $\hh^{2+\alpha}$ decay  at initial points,  we obtain $\hh^{2+\alpha}$ regularity in a full neighborhood of the parabolic boundary, see \cref{mn4}.

The paper is organized as follows: in \cref{section2}, we introduce  the relevant notions and gather some known results that are relevant to this  work and in \cref{section3}, we prove our main results.

\section{Notations, preliminaries and known results}\label{section2}
In this section, we introduce some basic notations and gather some preliminaries and known results.  A generic point in  the space time $\mathbb{R}^n \times \mathbb{R}$ will be  denoted by $(x,t)$, the euclidean ball in $\mathbb{R}^n$ of radius $r$ centered at $x$ will be denoted by $B_r(x)$ and  $Q_r(x,t)$ will denote  the parabolic cylinder of size $r$  in space time  defined by
\[
Q_r(x,t):= B_r(x) \times (t-r^2, t].
\]
When $(x,t)=(0,0)$, we  will often denote such a set by $Q_r$.  We will  use the notation  $Q^{+}_{r}$ to be the set $B^{+}_{r} \times (-r^{2},0]$ and by  $Q^{0}_{r}$ to be the set $B^{0}_{r}\times(-r^{2},0]$ where  $B^{+}_{r}:=\{x\in B_{r}~~|~~x_{n}> 0\}$ and $B^{0}_{r}:=\{x\in B_{r}~~|~~x_{n}=0\}$.  Hereafter,  the notation $S_n$ will indicate the space of $n \times n$ symmetric matrices.  The parabolic H\"older  spaces  will be denoted by $\hh^{k+\alpha}$ (see \cite[Chapter 4]{Li} for the details). 

\begin{rem}\label{boundary_definition}Since we are working with time dependent parabolic cylinders, we will follow the notation from \cite{Li} to define the different boundaries, $C\Om$ will be the corner boundary, $B\Om$ will be the bottom bounday, $S\Om$ to be the lateral bounday and $\mathcal{P}\Om$ to be the full parablic boundary. Since these definitions are quite standard, we shall refrain from describing them in detail and instead refer the reader to see  \cite[page 7, Chapter 2]{Li} for the precise descriptions. 
    \end{rem}

As mentioned in the introduction, in this paper  we consider the regularity  upto the boundary for the  following  Dirichlet-boundary value problem
\begin{equation*}
\begin{cases}
-u_{t}+F(D^{2}u,Du,x,t)=f(x,t)~\quad\text{in}~\Omega,
\\
u= g\quad \text{on $\mathcal{P}\Om$},
\end{cases}
\end{equation*}
where $\Omega\subset\R^{n+1}$ is a bounded smooth domain and  $\mathcal{P}\Om$ denotes the parabolic boundary of $\Omega$ (see \cref{boundary_definition}).  We remark that all functions considered in this paper are at least continuous in $\overline{\Omega}$.   Let us now define the  main structural hypothesis on $F$:
\begin{description}
\descitem{H1}{H1} There exist numbers $0< \lambda\leq\Lambda$,  and $K\geq0$ such that for any $(x,t)\in \overline{\Omega}$, $p,q\in\R^{n}$ and $M,N\in S(n)$, the following is satisfied:
\begin{equation*}
\mathcal{M}^{+}_{\lambda,\Lambda}(M-N)-K|p-q|\leq F(M,p,x,t)-F(N,q,x,t)\leq \mathcal{M}^{+}_{\lambda,\Lambda}(M-N)+K|p-q|.
\end{equation*}

\descitem{H2}{H2} There exist $\overline{\alpha}> 0$ and $\overline{C}> 0$ such that for all $M\in S(n)$ $p\in\R^{n}$ $(x,t),(y,s)\in\overline{\Omega}$ we have
\begin{equation*}
|F(M,p,x,t)-F(M,p,y,s)|\leq \overline{C}(|M|+|p|)(|x-y|+|t-s|^{1/2})^{\bar{\alpha}}.
\end{equation*}
\end{description}

We note that in \descref{H1}{H1},  $\mathcal{M}^{\pm}_{\lambda,\Lambda}$ denotes the Pucci extremal operators defined as
\begin{equation*}
\begin{cases}
\mathcal{M}^{+}_{\lambda, \Lambda}(M) = \Lambda \sum_{e_i >  0} e_i + \lambda \sum_{e_i <  0} e_i,
\\
\mathcal{M}^{-}_{\lambda, \Lambda}(M) = \lambda \sum_{e_i >  0} e_i + \Lambda \sum_{e_i <  0} e_i,
\end{cases}
\end{equation*}
where $\{e_i\}_{i=1}^n= \{e_i(M)\}_{i=1}^n$ indicate the eigenvalues of $M$.

\begin{rem}\label{notation_S}
    We shall use the standard notation $S(\lambda,\Lambda)$ to denote the class of continuous functions $v \in C^0$ which solves the following inequality in the viscosity sense:
    \begin{equation*}
        \mathcal{M}^{+}_{\lambda, \Lambda}(D^2 v) - v_t\  \geq \ 0 \ \geq \mathcal{M}^{-}_{\lambda, \Lambda}(D^2 v) - v_t.
    \end{equation*}
\end{rem}

\subsection{Known boundary \texorpdfstring{$\hh^{1+\alpha}$}. regularity result}
Since the techniques are perturbative in nature, we shall recall a $\hh^{1+\alpha}$ boundary regularity for viscosity solutions (see \cite{W} or \cite{Imbert} for the precise definition) of 
\begin{equation}\label{def1}
~~\left\{
\begin{aligned}{}
-v_{t}+\mathcal{M}^{+}_{\lambda,\Lambda}(D^{2}v)+K|Dv|&\geq-L~\quad\text{in}~Q^{+}_{1},\\
-v_{t}+\mathcal{M}^{-}_{\lambda,\Lambda}(D^{2}v)-K|Dv|&\leq L~\quad\text{in}~Q^{+}_{1},
\end{aligned}
\right.
\end{equation} 
where $K,L\geq0$ and   $\mathcal{M}^{\pm}_{\lambda,\Lambda}$ denote the extremal Pucci operators.  We now state  the first relevant result on $\hh^{1+\alpha}$ regularity at lateral  boundary for solutions to \cref{def1}, see for instance  \cite[Theorem 2.1]{W1}(See also \cite{MMW}). 
\begin{thm}\label{thm_deff}
Let $v\in C(\overline{Q^{+}_{1}})$ be a viscosity solution of \cref{def1}  with $v=g$ on $Q^{0}_1$ such that $g\in \hh^{1+\bar{\alpha}}(Q^{0}_{1})$,  for some $\bar{\alpha}> 0$.  Then there exists $\alpha\in(0,\bar{\alpha})$ and  a function  $G\in \hh^{\alpha}(Q^{0}_{1/2},\R^{n})$, called the "gradient" of $v$ at the boundary  such that the following holds:
\begin{equation*}
\|G\|_{\hh^{\alpha}(Q^{0}_{1/2})}\leq C(\|v\|_{L^{\infty}(Q^{+}_{1})}+L+\|g\|_{\hh^{1+\bar{\alpha}}(Q^{0}_{1})}).
\end{equation*}

Moreover for  every $(x_{0},t_{0})\in Q^{0}_{1/2}$ we have
\begin{equation*}
|v(x,t)-v(x_{0},t_{0})-\langle G(x_{0},t_{0}),x-x_{0}\rangle|\leq C(|x-x_0| + |t-t_0|^{1/2})^{1+\alpha},
\end{equation*}
whenever $(x,t) \in Q_{1/2}^+$ with $t \leq t_0$. Here $C= C(\|v\|_{L^{\infty}(Q^{+}_{1})},L,\|g\|_{\hh^{1+\bar{\alpha}}(Q^{0}_{1})}, n, \lambda, \Lambda) $.
\end{thm}

Now suppose  $u$ solves  the following Dirichlet problem 
\begin{equation}\label{dt1}
\left\{
\begin{aligned}{}
-u_{t}+F(D^{2}u,Du,x,t)&=f(x,t)~~\quad \text{in}~~Q^{+}_{1},\\
u&=g~~\quad\text{on}~~~Q^{0}_{1},
\end{aligned}
\right.
\end{equation}
where $F$ satisfies the structural assumptions  as in \descref{H1}{H1} and \descref{H2}{H2} and  $f$ is continuous.  It turns out that by a standard argument, we can transfer the regularity using \cref{thm_deff} coupled with interior $\hh^{1+\alpha}$ estimates as in \cite{W} to the viscosity solutions of \cref{dt1} which allow us to conclude that solutions to \cref{dt1} are in fact $\hh^{1+\alpha}$ upto the boundary. 
\begin{thm}\label{ub}
Let $u\in C(Q^{+}_{1}\cup Q^{0}_{1})$ be a viscosity solution of \cref{dt1} where $g\in \hh^{1+\bar{\alpha}}(Q^{0}_{1})$.  Then $u\in \hh^{1+\alpha}(\overline{Q^{+}_{1}})$ with a norm bounded by the quantity $\|u\|_{L^{\infty}(Q^{+}_{1})},K, \bar{\alpha},\lambda, \Lambda, ||f||_{L^{\infty}}$ and $\|g\|_{\hh^{1+\bar{\alpha}}(Q^{0}_{1})}$.  Moreover $\alpha$ depends on $\bar{\alpha}, \lambda, \Lambda$ and $n$. 
\end{thm}

%
%
%
%

\section{Proof of the main results}\label{section3}
The first step we show is the  existence of second order Taylor approximation  at a  boundary point for solutions to homogeneous equations which vanishes on a portion of a flat boundary. This is analogous to  \cite[Lemma 4.1]{SS}.

\begin{lem}\label{sl7}
Let $u$ satisfying $|u| \leq 1$ be a viscosity solution to
\begin{equation*}
\begin{cases}
F(D^2u, Du)= u_t\quad \text{in $B_1^{+} \times  (-1, 0]$},
\\
u=0\quad \text{on $B_1^0 \times (-1, 0]$},
\end{cases}
\end{equation*}
where $F$ satisfies  \descref{H1}{H1}. Then for some $\alpha = \alpha(n,\lambda,\Lambda) \in (0,1)$, there exists an $\alpha$-H\"older continuous function, $H: B_{1/2}^{0} \times (-1/4, 0] \to \R^{n \times n}$ with universal bounds (i.e., depending only on the structural conditions)  such that  for all $(x_0, t_0) \in B_{1/2}^{0} \times (-1/4, 0]$, we have
\begin{equation*}
F(H(x_0, t_0), Du(x_0, t_0))=0.
\end{equation*}
Moreover, there exists a universal constant $C= C(n,\lambda, \Lambda,K, F(0,0))$ such that for all $(x, t)$ such that $t \leq t_0$, the following boundary $\hh^{2 + \alpha}$ estimate holds:
\begin{equation}\label{sat1}
|u(x,t) - \langle Du(x_0, t_0), x-x_0\rangle  - \frac{1}{2} \langle  H(x_0, t_0) (x-x_0), x-x_0\rangle |  \leq C( |x-x_0| + |t-t_0|^{1/2})^{2+\alpha}.
\end{equation}
\end{lem}
\begin{proof}
We have from \cref{ub} that $u \in \hh^{1+\alpha}(\overline{Q_r^+})$ for all $r <1$.  Moreover by taking incremental quotients of the type
\[
 u_{h, e_i} (x, t) := \frac{ u(x+he_i, t)- u(x, t)}{h} \qquad \text{for} \ i = 1,2,\ldots, n,
\]
and by passing to the limit as $h \rightarrow 0$, we note that $v= u_i$ solves
\begin{equation*}
\begin{cases}
v \in \mathcal{S}(\lambda, \Lambda),
\\
v=0\ \  \text{on}\  B_1^0 \times (-1, 0].
\end{cases}
\end{equation*}
Recall the notation $S(\lambda, \Lambda)$ defined in \cref{notation_S}. In a similar way,  by taking repeated difference quotients of the type
\[
u_{h, t}(x,t)=\frac{u(x, t+h)- u(x, t)}{h^{k\alpha/2}} \qquad \text{for} \ \ k=1,2...,
\]
we observe  that $u_t \in \mathcal{S}(\lambda, \Lambda)$ and consequently,  $u_t$ is in  $ \hh^{\alpha}(\overline{B_r^+} \times (-r^2, 0])$ for all $r < 1$ and $u_t \equiv 0$ on $\{x_n=0\}$.  Now by applying \cref{thm_deff} to $u_i$, we obtain
\begin{equation}\label{s20}
|u_i(x, t) - \langle A_i(x_0, t_0), x- x_0\rangle  | \leq C( |x-x_0| + |t-t_0|^{1/2} )^{1+\alpha},
\end{equation}
for $(x_0, t_0) \in B_{3/4}^0 \times (-3/4, 0]$ and where $A_i$ is $\alpha$- H\"older continuous on $B_{3/4}^0 \times (-3/4, 0]$.  We now define
\[
H_{ij}=0\ \  \text{for} \  i, j=1, ..., n-1\qquad  \text{and}\qquad  H_{ni}= A_i\ \ \text{for} \  i=1, .., n-1.
\]
Moreover, by applying \cref{thm_deff} to $u_t$, we obtain
\begin{equation}\label{s2}
|u_t(x, t) - \langle \tilde{H}(x_{0},t_{0}), x- x_0\rangle  | \leq C( |x-x_0| + |t-t_0|^{1/2} )^{1+\alpha},
\end{equation}
for all $(x_{0},t_{0})\in B^{0}_{r}(x_{0})\times(-r^{2},0]$ and where $\tilde H$ is a H\"older continuous function on $\{x_n=0\}$.  

We now show that \cref{s2} implies  $u_n$ is $\hh^{1, \alpha}$ on $\{x_n=0\}$ and also that its tangential derivative coincides with $H_{ni}$.  Clearly, it suffices to establish the claim at the point $(0,0,0)$, which we prove as in  \cite{SS}.  On one hand, for $i=1, .., n-1$, we have 
\begin{equation*}
\begin{array}{rcl}
    u( he_i + ke_n, t) - u( k e_n, 0)  & \overset{\redlabel{35a}{a}}{=} &  h u_{i} ( ke_n + \theta h e_i, \theta t) + t u_t( k e_n + \theta h e_i, \theta t) \\
    & \overset{\redlabel{35b}{b}}{\leq} &  h k H_{ni}(\theta h e_i, \theta t) +  t k \tilde H(\theta h e_i, \theta t) + C|h| |k|^{1+\alpha} + C |t| |k|^{1+\alpha}\\
    & \overset{\redlabel{35c}{c}}{\leq} &  h k H_{ni} (0, 0, 0) +  t k \tilde H(0,0,0)  \\
    && +C(|k| |h||t|^{\alpha/2}  + |h||k|^{1+\alpha} + |k| |h|^{1+\alpha} +|t| |k| |h|^{\alpha} + |k| |t|^{\frac{1+\alpha}{2}} +  |t||k|^{1+\alpha}), 
\end{array}
\end{equation*}
where to obtain \redref{35a}{a}, we applied Mean value theorem for some $\theta < 1$, to obtain \redref{35b}{b}, we made use of \cref{s20} and \cref{s2} and finally to obtain \redref{35c}{c}, we made use of H\"older continuity of $H_{ni}$ and $\tilde H$.

On the other hand, we can obtain a lower bound for the difference $u( he_i + ke_n, t) - u( k e_n, 0)$ using mean value theorem in the $e_n$ direction along with  using the fact that $u \equiv 0$ on $\{x_n=0\}$, which gives
\begin{equation*}
\begin{array}{rcl}
u( he_i + ke_n, t) - u( k e_n, 0)& \overset{\redlabel{34a}{a}}{=} &  ku_n( he_i + \theta_1 k e_n, t)- ku_n (\theta_2 k e_n, 0)\\
 & \overset{\redlabel{34b}{b}}{\geq} &   k u_n (he_i,t) - ku_n(0,  0)   - Ck^{1+\alpha},
\end{array}
\end{equation*}
where to obtain \redref{34a}{a}, we used Mean Value theorem with $\theta_1,\theta_2 < 1$ and to obtain \redref{34b}{b}, we used the fact that $u_n \in \hh^{\alpha}$ upto $\{x_n=0\}$.

By combining the upper and lower bounds from   above followed with  dividing  by  $k$ and  letting  $k \to 0$, we obtain
\begin{equation*}
\begin{array}{rcl} 
u_n (he_i,t) - u_n(0, 0)- hH_{ni} & \leq &  C(|h|^{1+\alpha}  + |h| |t|^{\alpha/2} + |t|^{\frac{1+\alpha}{2}} + |t| |h|^{\alpha})
\\
 & \leq &   C(h^{1+\alpha} + |t|^{\frac{1+\alpha}{2}}),
\end{array}
\end{equation*}
where to obtain the last ineuqality, we applied Young's ineuqality  to $|t| |h|^{\alpha}$ and $|h| |t|^{\alpha/2}$. Analogously, the following lower bound  also holds:
\[u_n (he_i,t) - u_n(0, 0)- hH_{ni}\geq-( C|h|^{1+\alpha} + C|t|^{\frac{1+\alpha}{2}}).\] 
  Thus, combining both the bounds, we get
\begin{equation*}
|u_n (he_i,t) - u_n(0, 0)- hH_{ni}|\leq C(|h|^{1+\alpha} + |t|^{\frac{1+\alpha}{2}}),
\end{equation*}
which implies that $u_n$ is in $\hh^{1+\alpha}(B^{0}_{3/4}\times(-9/16,0])$ and $\partial_iu_n=H_{ni}$ on $B^{0}_{3/4}\times(-9/16,0]$ for $i=1,2,\ldots,n-1$.  Consequently we can let $H_{in}= H_{ni}$  on $B^{0}_{3/4}\times(-9/16,0]$.  At this point, we can finish the construction of $H$ by defining $H_{nn}(x_0,t_0)$ as the  unique real number for which $F(H(x_0,t_0), Du(x_0,t_0))=0$. The uniform ellipticity of $F$ ensures  that $H_{nn}$ is also  in $\hh^{\alpha}$. 


It now remains to prove \cref{sat1} which we do as follow: Without loss of generality we will show that the inequality is valid at $(x_0,t_0)=(0,0)$. First, let us  show that for some $r> 0$, if there exists $C_0$ such that 
\begin{equation}\label{s3}
u(0,r,0)-u_{n}(0,0,0)r-\frac{1}{2}H_{nn}(0,0,0)r^{2}=\pm C_{0}r^{2+\alpha},
\end{equation}
holds, then there exists a universal constant $C$ such that 
\begin{equation}\label{s31}
C_0 \leq C.
\end{equation}
 
 The proof is by contradiction, suppose on the contrary that is  not the case and  we have plus sign in \cref{s3} (minus sign will be treated analogously),  then consider the following auxiliary function
\begin{equation}\label{s4}
\begin{array}{rcl}
w(x',x_{n},t)&=&\langle Du(0,0), x \rangle+\frac{1}{2}\langle H(0,0)x,x\rangle+C_{0}r^{\alpha}x_{n}^2-MC_{1}r^{\alpha}|x|^{2}+C_{2}r^{\alpha}t\\
&=&u_{n}(0,0,0)x_{n}+H_{ni}x_{i}x_{n}+\frac{1}{2}H_{nn}x^{2}_{n}+C_{0}r^{\alpha}x^{2}_{n}-MC_{1}r^{\alpha}|x|^{2}+C_{2}r^{\alpha}t,
\end{array}
\end{equation}
where $C_{0}$ is  as in \cref{s3} and $C_{1}, C_{2}$ and $M$ are  constants to be chosen below. Let us also define
\begin{equation}\label{def_k}
k:=\max\{w(0,s,0)-u(0,s,0)~~|~~s\in[0,r]\}.
\end{equation}
Now our aim is to show that there exists some $r_{0}$ such that for all $0< r< r_{0}$ we have
\begin{enumerate}
\item\label{conc1} $w$ is a subsolution of 
\begin{equation*}
-w_{t}+F(D^{2}w,Dw)\geq0~~~\quad \text{in}\ \{|x'| <  r\} \times(0,r)\times(-r^{2},0],
\end{equation*}
\item\label{conc2} and it satisfies
\begin{equation}\label{s5}
w\leq u+k \quad \text{on} \ A_{1}\cup A_{2}\cup A_{3}\cup A_{4},
\end{equation}
where $A_{1}:=\{|x'|\leq r\}\times\{x_{n}=0\}\times(-r^{2},0)$,  $A_{2}:=\{|x'|=r\}\times\{0< x_{n}< r\}\times(-r^{2},0)$, $A_{3}:=\{|x'|\leq r\}\times\{x_{n}=r\}\times(-r^{2},0)$ and $A_{4}:=\{|x'|\leq r\}\times\{0 < x_{n}< r\}\times\{t=-r^{2}\}$.
\end{enumerate}
We start by showing \cref{conc2} first as follows:  Since $u(0,0,0)=w(0,0,0)=0$, from \cref{def_k}, we see that   $k\geq0$.  We then make the following observations. 
\begin{description}
\descitemnb{Observation 1}{one} From the definition of $w$ it follows that
\begin{equation*}
w(x',s,-r^{2})-w(0,s,0)=H_{ni}(0,0,0)x_{i}s-MC_{1}r^{\alpha}|x'|^{2}-C_{2}r^{2+\alpha}.
\end{equation*}
\descitemnb{Observation 2}{two} From \cref{s3} and \cref{s4}, we have
\begin{equation*}
w(0,r,0)-u(0,r,0)=-MC_{1}r^{2+\alpha}.
\end{equation*}
\descitemnb{Observation 3}{three} Furthermore by  mean value theorem, for some $\theta \in (0,1)$,  we have the following estimate
\begin{equation*}
\begin{array}{rcl}
u(x',s,t)-u(0,s,0)&\overset{}{=}& u_{i}(\theta x',s,\theta t)x_{i}+tu_{t}(\theta x',s,\theta t)\\ 
&\overset{\cref{s20}}{\geq}& H_{ni}(\theta x',0,\theta t)x_{i}s-C_{1}|x'|s^{1+\alpha}+tu_{t}(\xi,s,\theta t)\\
&\geq& H_{ni}(0,0,0)x_{i}s-C_{1}|x'|s(|x'|+|t|^{1/2})^{\alpha}-C_{1}|x'|s^{1+\alpha}+tu_{t}(\theta x',s,\theta t).\\
\end{array}
\end{equation*}
Now by applying the  \cref{s2}  to  the term $t u_t$  above and by using the H\"older continuity of $\tilde H$, we  can infer that the following estimate holds, 
\begin{equation}\label{ty}
\begin{array}{rcl}
u(x',s,t)-u(0,s,0) & \geq &  H_{ni}(0,0,0)x_{i}s-C_{1}s|x'|^{1+\alpha} -C_{1}|x'|s^{1+\alpha} \\
&& +C_{1}ts^{1+\alpha}- C_1 s |t|^{1+\alpha/2} - C_1 s |t| |x'|^{\alpha} + \tilde{H}(0,0,0)st.
\end{array}
\end{equation}
\end{description}

 We now verify \cref{conc2}. 
\begin{description}
    \item[On $A_1$, i.e., when  $\{x_n=0\}$] On this set  we have 
\[w(x',0,t)=-MC_{1}r^{\alpha}|x'|^{2}+C_{2}r^{\alpha}t\leq0=u(x',0,t)\leq u(x',0,t)+k.\]
    \item[On $A_2$, i.e., when  $\{|x'|=r\}$]  In this case,  using \cref{ty}, we have
\begin{equation*}
\begin{array}{rcl}
w(x',s,t)-u(x',s,t)-k & = & [w(x',s,t)-w(0,s,0)]-[u(x',s,t)-u(0,s,0)]-[k+u(0,s,0)-w(0,s,0)]\\
&\overset{\cref{def_k}}{\leq} & [w(x',s,t)-w(0,s,0)]-[u(x',s,t)-u(0,s,0)]\\
&\overset{\redlabel{a2a}{a}}{\leq} &  H_{ni}(0,0,0)x_{i}s-MC_{1}r^{\alpha+2}+C_{2}r^{\alpha}t\\
&& -\Big[H_{ni}(0,0,0)x_{i}s-C_{1}s|x'|^{1+\alpha} -C_{1}|x'|s^{1+\alpha}]\\
&& - [C_{1}ts^{1+\alpha}- C_1 s |t|^{1+\alpha/2} - C_1 s |t| |x'|^{\alpha} + \tilde{H}(0,0,0)st].
\end{array}
\end{equation*}
Now since, $|t|< r^2 \Rightarrow -r^2 < t < 0$ and $0 < s<r$, by choosing $M$ sufficiently large,  we can ensure that the quantity  appearing on the right hand side of \redref{a2a}{a}  is nonpositive.  Therefore,  on $A_2$, we see that  $w\leq u+k$.

    \item[On $A_3$, i.e., when  $\{x_n=r\}$] On this set, we have
\begin{equation*}
\begin{array}{rcl}
w(x',r,t)-u(x',r,t)-k& \overset{\redlabel{a3a}{a}}{\leq}& [w(x',r,t)-w(0,r,0)] -[u(x',r,t)-u(0,r,0)] + [ w(0, r, 0) - u(0, r, 0)]\\ 
&\overset{\redlabel{a3b}{b}}{\leq}& \left[H_{ni}x_{i}r-MC_{1}r^{\alpha}|x'|^{2}+C_{2}r^{\alpha}t-MC_{1}r^{2+\alpha} \right]\\
&& -\left[H_{ni}(0,0,0)x_{i}r-C_{1}r|x'|^{1+\alpha} -C_{1}|x'|r^{1+\alpha}+C_{1}tr^{1+\alpha}\right] \\
&& - \left[- C_1 r |t|^{1+\alpha/2} - C_1 r |t| |x'|^{\alpha} + \tilde{H}(0,0,0)rt\right],
\end{array}
\end{equation*}
where to obtain \redref{a3a}{a}, we used the fact that $k \geq 0$ and to obtain \redref{a3b}{b}, we made use of the fact that $w(0, r,0)- u(0, r, 0)= - M C_1r^{2+\alpha}$ along with \cref{ty}.
Again since $|x'| \leq r$ and  $|t| \leq r^2$, we can deduce that $w \leq u+ k$ on $A_3$ provided $M$ is  sufficiently large.

    \item[On $A_4$, i.e., when  $\{t=-r^2\}$] Finally in this case we have 
\begin{equation*}
\begin{array}{rcl}
w(x',s,-r^{2})-u(x',s,-r^{2})-k& \leq& [w(x',s,-r^{2})-w(0,s,0)]-[u(x',s,-r^{2})-u(0,s,0)]\\
&\overset{\cref{ty}}{\leq}&  H_{ni}x_{i}s-MC_{1}r^{\alpha}|x'|^{2}-C_{2}r^{2+\alpha}\\ 
&&-\left[H_{ni}(0,0,0)x_{i}s-C_{1}s|x'|^{1+\alpha} -C_{1}|x'|s^{1+\alpha}+C_{1}r^2s^{1+\alpha}\right]\\
&& -\left[- C_1 s r^{2+\alpha} - C_1 s r^2 |x'|^{\alpha} + \tilde{H}(0,0,0)sr^2\right].
\end{array}
\end{equation*}
Now  since $|s|, |x'| \leq r$, the conclusion $w \leq u+ k$  in this case, likewise follows by choosing $C_{2}$ large enough. 

\end{description}


Now similar to the calculations from  \cite[Section 4]{SS}, we find that
\begin{equation*}
-w_{t}+F(D^{2}w,Dw)\geq r^{\alpha}\big[-C_{2}+2\lambda(C_{0}-MC_{1})-2MC_{1}\Lambda(n-1)\big]-Cr.
\end{equation*}
Now having chosen $C_{1},C_{2}$ and $M$ sufficiently large which guarantees that \cref{s5} holds, we  can  now  choose $C_{0}$ sufficiently large to ensure that  for $0 <  r< 1$, the following holds
\begin{equation*}
-w_{t}+F(D^{2}w,Dw)\geq0\quad \text{in}\ Q_{r}\times(-r^2,0].
\end{equation*}
Thus, thanks to \cref{s5} and the fact that $w$ is a subsolution, we can now apply the  comparison principle to conclude that $w\leq u+k$ everywhere in $Q_{r}\times(-r^{2},0]$.

Now, if $k> 0$,  this means that $w(0,s,0)=u(0,s,0)+k$ for some $s\in(0,r)$,  which is a  contradiction to the strong maximum principle. On the other hand, if $k=0$,  this would then contradict H\"{o}pf  lemma, since $\partial_{n}(u-w)=0$ at the origin. Thus \cref{s3} and \cref{s31} holds and  for some $C$ universal (by translating the origin), we have
\begin{equation}\label{s80}
\left|u(x',x_n,t)-u_{n}(x',0,t)x_n-\frac{1}{2}H_{nn}(x',0,t)r^{2}\right| \leq Cx_n^{2+\alpha} \qquad  \text{for}\ \  0< r< 1,
\end{equation}
and for all $(x', 0, t) \in Q^{0} (1)$. Thus, \cref{s80} coupled with  the fact that $u_n$ restricted to $\{x_n=0\}$ is in $\hh^{1+\alpha}$ with the following estimates
\begin{equation*}
\begin{cases}
|u_n(x', 0, t) - u_n(0) - H_{ni} x_i| \leq C(|x'| +|t|^{1/2})^{2+\alpha},
\\
| H_{nn}(x', 0, t) - H_{nn}(0,0)| \leq C(|x|'+ |t|^{1/2})^{\alpha},
\end{cases}
\end{equation*}
implies the conclusion of the lemma. 
\end{proof}

We now state  and prove the relevant compactness lemma which allows the passage to more general non-homogeneous equations.


\begin{thm}\label{compact}
Suppose that $F$ satisfies the structural condition \descref{H1}{H1} with  $F(0,0,0,0)=0$  and  assume that it  has a modulus of continuity $\rho$ in the $(x,t)$ variable,  more precisely it satisfies 
\begin{equation*}
|F(M,p,x,t)-F(M,p,y,s)|\leq \overline{C}(|M|+|p|)\rho( (|x-y|+ |t-s|^{1/2}).
\end{equation*}
 Furthermore,  assume that  $\rho$ satisfies  $\rho(r)\leq\phi(r)$ for all $r\in[0,1)$ for another modulus of continuity $\phi$.  Now let   $u\in C(B^{+}_{1}\times(-1,0])$ with  $\|u\|_{L^{\infty}(B^{+}_{1}\times(-1,0])}\leq1$ be  a viscosity solutions of
\begin{equation*}
\left\{
\begin{aligned}{}
-u_{t}+F(D^{2}u,Du,x,t)&=0\quad \text{in}~~Q^{+}_{1},\\
u&=0\quad\text{on}~~~Q^{*}_{1}.
\end{aligned}
\right.
\end{equation*}
Then there exist universal constants $\delta, \gamma, \beta \in (0,1)$ such that if  $\rho(1) <  \delta$, then for some $A \in \mathbb{R}$ and  $H \in S^n$ solving
\begin{equation}\label{des1}
F(H, Ae_n, 0,0) =0,
\end{equation}
the following conclusion holds
\begin{equation}\label{des2}
||u- Ax_n - \frac{1}{2}\langle Hx,x\rangle ||_{L^\infty(Q_{\gamma}^+)} \leq \gamma^{2+\beta}.
\end{equation}
\end{thm}
\begin{proof}
%
The proof is by contradiction, suppose on the contrary,    for all $k\in\N$ there exists $F_{k}, u_k,  \rho_{k}$ satisfying the hypothesis of the lemma such that $\rho_{k}(1)< 1/k$,  $\rho_{k}(r)< \phi(r)$ and  $u_{k}\in C(B^{+}_{1}\times(-1,0])$ which solves 
\begin{equation*}
\begin{cases}
-(u_k)_{t}+F_{k}(D^{2}u_{k},Du_{k},x,t)&=0~~\text{in}~~Q^{+}_{1},\\
u_{k}=0~~\text{on}~~~Q^{*}_{1},
\end{cases}
\end{equation*}
the hypothesis \cref{des1} and \cref{des2} simultaneously doesn't hold for any choice of $A, H, \gamma, \beta$.

Since $\|u_k\|_{L^{\infty}(B^{+}_{1}\times(-1,0])}\leq1$, therefore by using the boundary $\hh^{\alpha}$-estimate we find that $\|u_{k}\|_{\hh^{\alpha}(\overline{Q_{1/2}^+)}}\leq C$. Consequently from  Arzela-Ascoli,  we can extract a subsequence of $\{u_k\}$ which converges uniformly to some $u_{0}$.  Also by using the equicontinuity of $F_k$, 's we note that upto a subsequence $\{F_k\}$'s  converges locally uniformly on $\Omega\times\R^{n}\times S(n)$ to some $F_{0}$.  Therefore, by using  the stability result for  viscosity solutions we obtain that $u_{0}$ solves 
\begin{equation*}
\left\{
\begin{aligned}{}
-(u_0)_{t}+F_{0}(D^{2}u_{0},Du_{0})&=0~~\text{in}~~Q^{+}_{1},\\
u_{0}&=0~~\text{on}~~~Q^{*}_{1}.
\end{aligned}
\right.
\end{equation*}
Since $F_{0}$ satisfies all the assumptions of \cref{sl7},  thus for some $A_0, H_0$ we have that 
\begin{equation*}
\begin{cases}
F_0( H_0, A_0)=0,
\\
\left|u_0(x,t) - A_0 x- \frac{1}{2}\langle  H_0x, x\rangle \right| \leq C (|x| + |t|^{1/2})^{2+\alpha}.
\end{cases}
\end{equation*}
Since $F_k \to F_0$ locally,  therefore we  can ensure that for some $a_k$ with $|a_k| = o(1)$ and   $H_k= H_0  + a_k e_n \otimes e_n $  solves
\begin{equation}\label{b0}
F_k( H_k, A_0, 0, 0)=0.
\end{equation}
This can be seen as follows, since $F_0( H_0, A_0)=0$, therefore we must have $F_k ( H_0, A_0,0,0)= o(1)$. 
Now from the uniform ellipticity of $F_k$ we  note that  for any $a$ we have 
\begin{equation}\label{com1}
o(1) -  \Lambda |a| \leq F_k( H_0 + a  e_n \otimes e_n, A_0, 0, 0) \leq o(1) +  \Lambda |a|.
\end{equation}
Thus from \cref{com1}, we  can deduce that for some $a_k$ with $|a_k|=o(1)$, estimate \cref{b0} holds. 
We now choose some $\beta <  \alpha$ and corresponding to such a $\beta$, let $\gamma \in (0, 1/2)$ be  such that
\begin{equation}\label{b1}
C (2\gamma)^{2+\alpha}= \frac{\gamma^{2+\beta}}{2}.
\end{equation}
Now since  $u_k \to u_0$ uniformly on $\overline{Q_{1/2}^+}$,  therefore using \cref{b0} and \cref{b1} we see that for large enough $k's$, corresponding to  $H_k$ and $A_k=A_0$, both \cref{des1} and  \cref{des2} holds which is a contradiction for all such $k's$. 
\end{proof}

\subsection{\texorpdfstring{$\hh^{2+\alpha}$}. differentiability at lateral boundary points}
With the compactness lemma in hand, we  now establish the $\hh^{2+\alpha}$  regularity at  lateral boundary  points for more general equations. 
\begin{thm}\label{mn1}
Suppose that $F$ satisfies $\descref{H1}{H1}$ and $\descref{H2}{H2}$ and  let $\Omega$ be a $\hh^{2+\bar{\alpha}}$ domain and $f \in \hh^{\bar{\alpha}}(\overline{\Omega})$.  Let $u$ be a viscosity solution of
\begin{equation*}
-u_{t}+F(D^{2}u,Du,x,t)=f(x,t)~~~\text{in}~\Omega \cap Q_2,
\end{equation*}
such that $u|_{\mathcal{P}\Om}=g\in \hh^{2+\bar{\alpha}}(\mathcal{P}\Om)$.  Assume that $S\Omega \cap Q_1 = \mathcal{P}\Om \cap Q_1$ and  also that  $(0,0) \in   \mathcal{P}\Om$. Then there exists $\beta \in (0,1)$ depending on $\bar{\alpha}, n, \lambda, \Lambda$ ,  a function $H\in \hh^{\beta}(\mathcal{P}\Om \cap Q_{1/2})$ which behaves like the "spatial hessian" of $u$ and  a function $H_1 \in \hh^{\beta}(\mathcal{P}\Om \cap Q_{1/2})$ which behaves like the "time derivative" of $u$  on $\mathcal{P}\Om$ such that there holds:
\begin{equation*}
F(H(x_{0},t_{0}),Du(x_{0},t_{0}),x_{0},t_{0}) - H_1(x_0, t_0) =f(x_{0},t_{0})\qquad \text{for~each} \ \ (x_{0},t_{0})\in  S\Omega \cap  Q_{1/2}.
\end{equation*}
Moreover the following estimate holds:
\begin{equation}\label{des}
\begin{array}{l}
\left|u(x,t)-u(x_{0},t_{0})-\langle Du(x_{0},t_{0}), (x-x_{0})\rangle -\frac{1}{2}\langle H(x_{0},t_{0})(x-x_{0}),(x-x_{0})\rangle - H_1(x_0,t_0) (t- t_0)\right|\\
\hfill \leq CW(|x-x_{0}|+|t-t_{0}|^{1/2})^{2+\beta},
\end{array}
\end{equation}
where
\[ W:=\|u\|_{L^{\infty}(\Omega)}+\|f\|_{\hh^{\bar{\alpha}}(\Omega)}+\|g\|_{\hh^{2+\bar{\alpha}}(\partial_p\Omega)}.\] 
\end{thm}

\begin{proof}
We make the following reductions:
\begin{description}
\item[Reduction 1] As the domain $\Omega\in \hh^{2+\bar{\alpha}}$,  for any $(x_{0},t_{0})\in S\Omega$ there exists  $r> 0$ and $h\in \hh^{2+\alpha}{Q'_{r}(0)}$ such that
$\Omega\cap Q_{r}(x_{0},t_{0})=\{(x,s)\in Q_{r}~~|~~x_{n}> h(x',s)\}$.  Now let us define new function $v(y,t)=u(x,s)$ where $y'=x'$,  $y_{n}=x_{n}-h(y',s)$ and $s=t$.  It is easy to see that (after a rescaling),   $v$ solves the following Dirichlet problem
\begin{equation*}
\left\{
\begin{aligned}{}
-v_{s}+\tilde{F}(D^{2}v,Dv,y,s)&=\tilde{f}(y,s)~~\text{in}~~Q^{+}_{1},\\
v&=\tilde{g}~~\text{on}~~~Q^{*}_{1},
\end{aligned}
\right.
\end{equation*}
where $\tilde{f}(y,s)=f(x,s)+h_{s}(y',s)v_{n}(y,s)$, $\tilde{g}(y,s)=g(x,s)$ and 
\begin{equation*}
\tilde{F}(M,p,y,s)=F(D\psi^{T}(x)MD\psi(x)+[p\cdot\psi_{i,j}(x)]_{1\leq i,j\leq n},   \langle p,  D\psi(x)\rangle ,x,s),
\end{equation*}
where $y=\psi(x)$. It is easy to observe that $\tilde{F}$ satisfies the structural conditions $\descref{H1}{H1}$ and $\descref{H2}{H2}$.  Furthermore, if $F(0,0,0,0)=0$ then $\tilde{F}(0,0,0,0)=0$.  Therefore, it is enough to prove the theorem in $Q^{+}_{1}$ with boundary condition on $Q^{0}_{1}$. 
\item[Reduction 2]~ Now by taking a $\hh^{2+\alpha}$-extension of $g$ on $\Omega$ followed  by letting $v=u-g$,  we can assume that $g=0$ on $Q^{0}_{1}$.  
\item[Reduction 3] Since it is already known that $Du$ is Holder continuous upto the boundary, therefore, by letting $\tilde{F}(M,x,t)=F(M,Du(x,t),x,t)$, we may as well assume that $F$ doesn't depend on the gradient.
\item[Reduction 4] Also if we let $v=u+Cx^{2}_{n}$,  we find that $v=0$ on $Q^{0}_{1}$ and solves the following
\begin{equation*}
\tilde{F}(D^{2}v,x,t)=\tilde{f}\quad \text{in}\ \ Q^{+}_{1}, 
\end{equation*}
where $\tilde{F}(M,x,t)=F(M-2Ce_{n}\otimes e_{n},x,t)-f(0,0)$ and $\tilde{f}(x,t)=f(x,t)-f(0,0)$.  In view of uniform ellipticity of $F$, we can choose a suitable value of $C$ such that $\tilde{F}(0,0,0)=0$. By an abuse of notation,  we continue denoting $\tilde{F}$ by $F$ and $\tilde{f}=f$. Thus without loss of generality we can assume that $F(0,0,0)=0$ and $f(0,0)=0$. 
\item[Reduction 5]  We can also  assume that $\|u\|_{L^{\infty}(Q^{+}_{1})}\leq1$. 
\item[Reduction 6] Rescaling  $u$ as follows,  $u_{r}(x,t)=u(rx,r^{2}t)$,  we see that $u_r$ solves  
 \begin{equation*}
-(u_{r})_{t}+F_{r}(D^{2}u_{r}, x,t)=f_{r}(x,t),
\end{equation*}
 where $F_{r}(M,x,t)=r^{2}F(M/r^{2},rx,r^{2}t)$ and $f_{r}(x,t)=r^{2}f(rx,r^{2}t)$.  It is important to notice that  both $F_r$ and $F$ have the same ellipticity. Moreover, the modulus of continuity $\rho_{r}$ of $\tilde{F}_{r}(M,x,t)-f_{r}(x,t)$ and $\rho$ of $\tilde{F}=F(M,x,t)-f(x,t)$ are related in the following such that $\rho_{r}(s)\leq  \rho(rs)$ holds for all $0< s< 1$.  Therefore, if we choose $r$ sufficiently small, then we can ensure that  $\rho_{r}(1)< \delta$,  where $\delta$ is as in the hypothesis  of  \cref{compact}. From here onwards we will work with $u_r$ instead of $u$ and continue denoting it by $u$ itself. 
\end{description}

Without loss of generality,  it suffices to establish the estimate \cref{des}  at the point $(x_0, t_0)=(0,0)$.  Also in order to prove such an estimate,  it suffices to  show that there exists   $\gamma \in (0,1)$, such that  corresponding to $r_k= \gamma^k$ for a given $k \in \mathbb{N}$,  there exists  $A_{k}\in\R$ and $H_{k}\in S(n)$ such that the following holds:
\begin{equation}\label{itr}
\begin{array}{l}
|A_{k}-A_{k+1}|\leq Cr^{1+\beta}_k,\\
 F(H_k, 0,0)=0,\\
|H_{k}-H_{k+1}|\leq Cr^{\beta}_{k},\\
|u(x,t)-A_{k}x_{n}-\langle H_{k}x,x\rangle|\leq r^{2+\beta}_{k},~~\text{for $(x,t) \in Q_{r_k}^+$}.
\end{array}
\end{equation}

Then by a standard real analysis argument, the desired estimate at $(0,0)$ follows with $H(0,0)= \lim_{k \to \infty} H_k$.  Taking $\gamma, \beta$  as in  \cref{compact},  without loss of generality, we may assume that $\beta <  \bar{\alpha}$. 
Since $\|u\|_{L^{\infty}(Q^{+}_{1})}\leq1$, we observe that for $k=0$, the choice $A_0=0$ and $H_0=0$ works. Now assume that \cref{itr} holds upto some  $k$ and then we show that this implies the validity of \cref{itr} for $k+1$.  For $(x,t)\in Q^{+}_{1}$, let 
\begin{equation*}
u_{k}(x,t)=\frac{u(r_{k}x,r^{2}_kt)-A_{k}r_{k}x_{n}-\langle H_{k}r_{k}x,r_{k}x\rangle}{r^{2+\beta}_{k}}. 
\end{equation*}
.  

 In view of the validity of the  $k$-th  step, we have  that $\|u_{k}\|_{L^{\infty}(Q^{+}_{1})}\leq 1$.  Furthermore,  $u_{k}$ solves
\begin{equation*}\left\{
\begin{aligned}{}
-(u_k)_{t}+F_{k}(D^{2}u_{k},x,t)&=f_{k}(x,t)\quad \text{in}~~Q^{+}_{1},\\
u_{k}&=0\quad \text{on}~~Q^{0}_{1},
\end{aligned}\right.
\end{equation*}
where
\begin{equation*}
F_{k}(M,x,t)=\frac{1}{r^{\beta}_{k}}F(r^{\beta}_{k}M+H_{k},r_{k}x,r^{2}_{k}t)\quad \text{and}\quad f_{k}(x,t)=\frac{1}{r^{\beta}_{k}}f(r_{k}x,r^{2}_{k}t).
\end{equation*}
Now since  $F$,  $f$ are h\"older continuous with exponent $\bar{\alpha}$,  $F(H_k,0,0)=f(0,0)=0$ and  $\beta< \bar{\alpha}$,  we can thus  ensure that the modulus of continuity $\rho_k$ of $(F_k-f_k)$ satisfies all the conditions of   \cref{compact}. Consequently by   \cref{compact}, we have that  there exists  constants $A$ and  $H$ satisfying the conclusion of    \cref{compact} such that there holds
\begin{equation*}
\begin{cases}
||u_k - Ax_n - \frac{1}{2} \langle  Hx, x\rangle ||_{L^{\infty}(Q_{\gamma}^+)} \leq \gamma^{2+\beta},\\
F_k(H, 0, 0)=0.
\end{cases}
\end{equation*}

Scaling back to $u$, we get that the last estimate in \cref{itr} holds for $k+1$ with $H_{k+1}= H_k + \gamma^{k\beta} H$ and $A_{k+1}= A_k + \gamma^{k(1+\beta)} A$.  Since $F_k(H,0,0)=0$, it follows that $F(H_{k+1}, 0, 0)=0$.  The other assertions in \cref{itr} are also easily verified. This proves the induction step and the conclusion follows.
\end{proof}

\subsection{\texorpdfstring{$\hh^{2+\alpha}$}. regularity in a  neighborhood of the boundary}
In this subsection,  analogous to \cite[Theorem 1.3]{SS}, we prove  $\hh^{2+\alpha}$ regularity in a small neighborhood of the boundary under the additional assumption that $F$ is continuously differentiable in $M$, more precisely in addition to \descref{H1}{H1} and \descref{H2}{H2}, we assume that 
\begin{equation}\label{as1}
||D_M F(M_1, p,x,t) - D_M F(M_2, p, x, t)|| \leq \omega(||M_1- M_2||),
\end{equation}
is satisfied for some modulus of continuity $\omega$.  In order to prove this regularity, we first  state a certain generalization of the flatness result of  Yu Wang (see \cite{Wa})  as established    in  \cite[Theorem 1]{dP}. As mentioned in the introduction,   in the elliptic case, such a result was first obtained by Savin in \cite{sa}. 

\begin{thm}\label{flatness}
Let $F$ satisfy the structure conditions as in \descref{H1}{H1}, \descref{H2}{H2} and \cref{as1} and let $u$ be a viscosity solution to 
\begin{equation*}
F(D^2u,Du, x, t) - u_t= f(x,t)\quad \text{in}\ \  Q_{r}(x_0, t_0).
\end{equation*}
Moreover, assume that the following normalization holds,
\begin{equation}\label{norm}
 F(0,0, x_0, t_0) = f(x_0, t_0).
 \end{equation}
 Then there exists a universal $\delta, \alpha_0> 0$ such that if 
 \begin{equation}\label{smallness}
 \sup_{Q_r(x_0, t_0)} |u| \leq \delta' r^2,\ \text{for some $\delta' \leq \delta$}
 \end{equation}
 then $u \in \hh^{2+\alpha_0} (Q_{r/2})$ with the following quantitative estimate
 \[
 ||u||_{\hh^{2+\alpha_0}(Q_{r/2})} \leq C \frac{\delta'}{r^{\alpha_0}}.
 \]
\end{thm}

Combining  the $\hh^{2+\alpha}$ decay estimate as in  \cref{mn1} with the flatness result stated above, we   have the following result on the twice differentiability near the lateral boundary. Let us define the following notation $\Omega_{\delta_0}:= \{(x,t) \in \Omega: d((x,t), \mathcal{P}\Om) <  \delta_0\}$, where  $d((x,t), \mathcal{P}\Om)$ denotes the parabolic distance of the point $(x,t)$ from the parabolic boundary $\mathcal{P}\Om$. 
\begin{thm}\label{mn2}
Under the   hypothesis of  \cref{mn1} with the additional assumption  that $F$ satisfies \cref{as1}, we have that  there exists universal constants $\delta_0, \beta_0> 0$, such that $u \in \hh^{2+\beta_0} (\overline{\Omega_{ \delta_0} \cap Q_{1/2}})$.

By a standard covering argument, it further follows that $u \in \hh^{2+\beta_0} (\overline{\Omega_{ \delta_0} \cap Q_r})$ for all $r < 1$. 
\end{thm}

\begin{proof}
As in the previous proof, without loss of generality, we can assume that $g=0$, the boundary is flat and we can write $F(M,x,t)$ instead of $F(M,p, x, t)$.  Let $(x_0, t_0)$ be a point in $\Omega_{\delta_0} \cap Q_{1/2}$, then from the definition of such a set, there exists a point $(x_1,t_1) \in \mathcal{P}\Om \cap Q_{1/2}$ such that  $d:= |x_0- x_1| + |t_0-  t_1|^{1/2} <  \delta$.  Setting  $r= \frac{d}{4}$, we see that  $Q_r( x_0, t_0) \in \Omega$. Moreover by applying  \cref{mn1} to the boundary point $(x_1, t_1)$, we know that there exists a polynomial  $P(x,t) = Ax_n + \frac{1}{2} \langle  Hx, x\rangle $, such that
\begin{equation}\label{de1}
|u - P| \leq C r^{2+\beta}\quad  \text{in} \ Q_r(x_0, t_0),
\end{equation}
holds for some $\beta <  \bar \alpha$ and 
\begin{equation*}
F(D^2 P, x_1, t_1) = f(x_1, t_1).
\end{equation*}

Now thanks to \descref{H2}{H2} and the fact that $f \in \hh^{\bar \alpha}$, we have that 
\begin{equation}\label{co2}
F(D^2P, x_0, t_0) - f(x_0, t_0) = O(r^{\bar\alpha}).
\end{equation}
Now from the uniform ellipticity of $F$ and \cref{co2}, we have that with $\tilde P = P +a |x-x_0|^2$, there holds
\begin{equation*}
O(r^{\bar\alpha}) -  n \Lambda |a|\leq F(D^2 \tilde P, x_0, t_0) \leq O(r^{\bar\alpha}) + n \Lambda |a|.
\end{equation*}
Thus there exists a constant $a$  with $|a| = O(r^{\bar \alpha})$ such that  with $\tilde P = P +a |x-x_0|^2$, we have
\begin{equation}\label{co5}
F(D^2 \tilde P, x_0, t_0) = f(x_0, t_0).
\end{equation}

Thus $v = u- \tilde P$ solves 
\[
\tilde F(D^2v, x, t) = f(x, t)\quad \text{in} \ Q_r(x_0, t_0),
\]
where $\tilde F(M, x, t) = F( D^2 \tilde P +M, x, t)$. Also thanks to \cref{co5}, we have that 
\[
\tilde F(0, x_0, t_0) = f(x_0, t_0),
\]
i.e., the normalization condition \cref{norm} as in  \cref{flatness} holds. Moreover since $|a|= O(r^{\bar \alpha})$ and $\bar \alpha >  \beta$, it follows from \cref{de1} and triangle inequality that
\begin{equation*}
|u - \tilde P| \leq C r^{2+\beta}= Cr^{\beta} r^2 \quad \text{in}\  Q_r(x_0, t_0),
\end{equation*}
for a different constant $C$. Now the smallness assumption in \cref{smallness} is verified provided $C \delta_0^{\beta} \leq \delta$. Hence $u \in \hh^{2+\beta_0}(Q_{r/2}(x_0, t_0))$ for some $\beta_0> 0$ by  \cref{flatness}. Without loss of generality we may assume that $\beta_0 < \beta$.  Now in a standard way, one can put together this interior regularity result with the boundary result from  \cref{mn1} above to conclude that $u \in \hh^{2+\beta_0}(\overline{\Omega_{\delta_0} \cap Q_{1/2}})$. 
\end{proof}

\subsection{Full parabolic boundary regularity for time independent domains}
We now show that  \cref{mn2} coupled with a regularity result at initial points due to Lihe Wang as in \cite{W1}  allow us to conclude that under a similar compatibility condition, the solution to the fully nonlinear parabolic equation is in fact continuously twice differentiable in a  neighborhood of the whole parabolic boundary in time independent domains.   To describe the compatibility condition, let us look at the linear case: Let $\Omega = D \times (0, T)$ where $D$ is a $C^{2, \bar \alpha}$ domain in $\mathbb{R}^n$.  Note that for a  solution $u$  to the   classical Dirichlet  problem
\begin{equation*}
\begin{cases}
\Delta u -  u_t =0,
\\
u= g\quad \text{on}  \ \mathcal{P}\Om= D \times \{0\} \cup \partial D \times [0, T],
\end{cases}
\end{equation*}
to be continuously twice  differentiable upto the corner points $\partial_c \Omega= \partial D \times \{0\}$, one requires the following compatibility condition $g_t = \Delta g$ to be satisfied.

In the fully nonlinear setting, the result reads as follows:
\begin{thm}\label{mn4}
Let $u$ be a solution to 
\begin{equation*}
\begin{cases}
F(D^2u, Du, x, t) - u_t = f(x, t)\quad \text{in}\  D \times (0,T],\\
u=g\quad \text{on}\  \partial_p D \times (0, T],
\end{cases}
\end{equation*}
where $F,f,g$ satisfies the assumptions  as in  \cref{mn2} and  $D$ is a $C^{2, \bar \alpha}$ domain in $\mathbb{R}^n$.  Furthermore, for some $\hh^{2+\bar \alpha}$ extension  $\tilde g$ of $g$,  assume that at any point $(x_0, t_0) \in \partial D \times \{0\}$,  the following compatibility condition holds 
\begin{equation}\label{compt}
F(D^2\tilde g, D\tilde g, x_0, t_0) - \tilde g_t = f(x_0, t_0).
\end{equation}
Then   there exists $\delta_0, \beta_0> 0$ such that $u \in \hh^{2+\beta_0}(D_{\delta_0} \times [0, T] \cup D \times [0, \delta_0^2))$ where $D_{\delta_0}= \{x \in D: d(x,D) <  \delta_0\}$.

\end{thm}

\begin{proof}
Note that by our previous result  \cref{mn2}, we already have that $u$ is continuously twice differentiable in a neighborhood of the lateral boundary. Therefore, it suffices to look near  the initial points, i.e. regularity near the points in  $\overline{D} \times \{0\}$. Now thanks to \cref{compt}, we have that $v=u - \tilde g$ satisfies
\begin{equation*}
~~\left\{
\begin{aligned}{}
-u_{t}+\mathcal{M}^{+}_{\lambda,\Lambda}(D^{2}v)+K|Dv|&\geq h(x,t)~\text{in}~Q^{+}_{1},\\
-u_{t}+\mathcal{M}^{-}_{\lambda,\Lambda}(D^{2}v)-K|Dv|&\leq h(x,t)~\text{in}~Q^{+}_{1},
\end{aligned}
\right.
\end{equation*}
where $h \in \hh^{\bar{\alpha}}$ and $h \equiv 0$ at the corner points $\partial D \times \{0\}$. Hence the consistency condition as in  \cite[Theorem 2.14]{W1} is satisfied and we can thus ensure that $u-\tilde g$  has a  $\hh^{2+\beta}$  decay at  every point in $\overline{D} \times \{0\}$.  Using such a decay, we can repeat the arguments as in the proof of  \cref{mn2}    and obtain $\hh^{2+\beta_0}$ regularity near the initial boundary   using the flatness result  from \cref{flatness}.  The conclusion finally follows by a standard covering argument.
\end{proof}


\begin{thebibliography}{10}

\bibitem{AS}
Scott~N. Armstrong and Luis Silvestre.
\newblock Unique continuation for fully nonlinear elliptic equations.
\newblock {\em Math. Res. Lett.}, 18(5):921--926, 2011.

\bibitem{Ba}
Agnid Banerjee.
\newblock A note on the unique continuation property for fully nonlinear
  elliptic equations.
\newblock {\em Commun. Pure Appl. Anal.}, 14(2):623--626, 2015.

\bibitem{Ca}
Luis~A. Caffarelli.
\newblock Interior a priori estimates for solutions of fully nonlinear
  equations.
\newblock {\em Ann. of Math. (2)}, 130(1):189--213, 1989.

\bibitem{CC}
Luis~A. Caffarelli and Xavier Cabr\'{e}.
\newblock {\em Fully nonlinear elliptic equations}, volume~43 of {\em American
  Mathematical Society Colloquium Publications}.
\newblock American Mathematical Society, Providence, RI, 1995.

\bibitem{dP}
Jo\~{a}o~V\'{\i}tor da~Silva and Disson dos Prazeres.
\newblock Schauder type estimates for ``flat'' viscosity solutions to
  non-convex fully nonlinear parabolic equations and applications.
\newblock {\em Potential Anal.}, 50(2):149--170, 2019.

\bibitem{dT}
Disson dos Prazeres and Eduardo~V. Teixeira.
\newblock Asymptotics and regularity of flat solutions to fully nonlinear
  elliptic problems.
\newblock {\em Ann. Sc. Norm. Super. Pisa Cl. Sci. (5)}, 15:485--500, 2016.

\bibitem{Ev}
Lawrence~C. Evans.
\newblock Classical solutions of fully nonlinear, convex, second-order elliptic
  equations.
\newblock {\em Comm. Pure Appl. Math.}, 35(3):333--363, 1982.

\bibitem{Imbert}
Cyril Imbert and Luis Silvestre.
\newblock An introduction to fully nonlinear parabolic equations.
\newblock In {\em An introduction to the {K}\"{a}hler-{R}icci flow}, volume
  2086 of {\em Lecture Notes in Math.}, pages 7--88. Springer, Cham, 2013.

\bibitem{Kr}
N.~V. Krylov.
\newblock Boundedly nonhomogeneous nonlinear elliptic and parabolic equations
  in the plane.
\newblock {\em Uspehi Mat. Nauk}, 24(4 (148)):201--202, 1969.

\bibitem{Kr1}
N.~V. Krylov.
\newblock Boundedly inhomogeneous elliptic and parabolic equations in a domain.
\newblock {\em Izv. Akad. Nauk SSSR Ser. Mat.}, 47(1):75--108, 1983.

\bibitem{Li}
Gary~M. Lieberman.
\newblock {\em Second order parabolic differential equations}.
\newblock World Scientific Publishing Co., Inc., River Edge, NJ, 1996.

\bibitem{MMW}
F Ma, D. Moreira and L. Wang, 
\newblock Differentiability at lateral boundary for fully nonlinear parabolic equations, 
\newblock J. Differential Equations (263)
   no.~5, 2672-2686, 2017.
\bibitem{NV}
Nikolai Nadirashvili and Serge Vl\u{a}du\c{t}.
\newblock Singular solutions of {H}essian elliptic equations in five
  dimensions.
\newblock {\em J. Math. Pures Appl. (9)}, 100(6):769--784, 2013.

\bibitem{sa}
Ovidiu Savin.
\newblock Small perturbation solutions for elliptic equations.
\newblock {\em Comm. Partial Differential Equations}, 32(4-6):557--578, 2007.

\bibitem{SS}
Luis Silvestre and Boyan Sirakov.
\newblock Boundary regularity for viscosity solutions of fully nonlinear
  elliptic equations.
\newblock {\em Comm. Partial Differential Equations}, 39(9):1694--1717, 2014.

\bibitem{SS1}
Luis Silvestre and Boyan Sirakov,
\newblock Overdetermined problems for fullly non linear equations,
\newblock {\em Calc. Var. Partial Differential Equations}, 54(1): 989-1007, 2015. 
\bibitem{W}
Lihe Wang.
\newblock On the regularity theory of fully nonlinear parabolic equations. {I}.
\newblock {\em Comm. Pure Appl. Math.}, 45(1):27--76, 1992.

\bibitem{W1}
Lihe Wang.
\newblock On the regularity theory of fully nonlinear parabolic equations.
  {II}.
\newblock {\em Comm. Pure Appl. Math.}, 45(2):141--178, 1992.

\bibitem{Wa}
Yu~Wang.
\newblock Small perturbation solutions for parabolic equations.
\newblock {\em Indiana Univ. Math. J.}, 62(2):671--697, 2013.

\end{thebibliography}

\end{document}